\documentclass[12pt]{article}
\usepackage{amssymb,amsmath}
\usepackage{graphicx,float,color}
\usepackage[encapsulated]{CJK}

\usepackage{geometry} 
\usepackage{amsthm}
\usepackage{array} 
\usepackage{paralist} 
\usepackage{verbatim} 
\usepackage{subfig} 
\usepackage{amsmath}
\usepackage[all]{xy}
\newtheorem{theorem}{Theorem}[section]\newtheorem {lemma}{Lemma}[section]

\usepackage{sectsty}
\allsectionsfont{\sffamily\mdseries\upshape} 

\title{ Hausdorff dimension of the sets of Li-Yorke pairs for some chaotic dynamical systems including $A$-coupled expanding systems}
\author{ Hyonhui Ju$^1$, Jinhyon Kim$^1$, Peter Raith$^2$}
\date{} 

\begin{document}
\maketitle

\centerline{\small $^1$ Department of Mathematics, 
Kim Il Sung University, Pyongyang, D.P.R.Korea}
\centerline{\small $^2$ Department of Mathematics, University of Vienna, Austria}

\begin{abstract}
 In this paper we consider Hausdorff dimension of the sets of Li-Yorke pairs for some chaotic dynamical systems including $A$-coupled expanding systems. We prove that Li-Yorke pairs of $A$-coupled-expanding system under some conditions have full Hausdorff  dimension in the invariant set. Moreover we give a generalization of the result of [5] which is on the  Hausdorff dimension of Li-Yorke pairs of dynamical systems topologically conjugate to the full shift and have a self-similar invariant set, to the case of the dynamical systems topologically semi-conjugate to some kinds of subshifts. Further more we count Hausdorff dimension of ``chaotic invariant set" for some kinds of A-coupled-expanding maps.

\end{abstract}


\section{Introduction} 
 The term ``chaos" was introduced firstly into mathematics in the paper of Li-Yorke [3] that is based on the existence of Li-Yorke pairs.  Li-Yorke pairs are the pairs of points that approach each other for some sequence of moments in the time evolution and that remain separated for other sequences of moments. 
  In [5]  was discussed on the Hausdorff dimension of the set of Li-Yorke pairs for some simple classical chaotic dynamical systems.  It showed that, if the dynamical system in its invariant set is topologically conjugate to the full shift symbolic dynamical system and its invariant set is self-similar  or a product of self-similar sets, then its Li-Yorke pairs have full Hausdorff dimension in the invariant set. This result can be applied to simple classical models of ``chaotic" dynamics like the tent map, the Bakers transformation, Smales horseshoe, and solenoid-like systems (see [5]) since these kinds of systems have invariant sets of self-similar  or a product of self-similar sets in which the systems are topologically conjugate to full shift. To prove that Li-Yorke pairs have full dimension for more general hyperbolic systems
could be a task for further research[5], which is one topic we are going to study in this paper.

On the other hand,  the coupled-expanding and $A$-coupled-expanding with a transitive matrix $A$ has been recognized as one of the important criteria of chaos, see e.g.[8, 9, 10, 11, 12, 13]. There were obtained some results on the chaotic properties of $A$-coupled-expanding map  including chaos in the sense of  Li-Yorke and   Devaney,  and its topological entropy. However we can not find yet any result on dimensional-theoretical research for it.
In due consideration of the fact that Hausdorff dimension is a measure of the ``size", more exactly the ``thickness" of a set, we can say that the bigger  Hausdorff dimension of a so-called ``chaotic sets" as like the set of  Li-Yorke pairs, the more  the chaotic behave of the system occurs. 
Therefore it  seems natural and meaningful  to investigate on the Hausdorff dimension of the set of Li-Yorke pairs for the $A$-coupled-expanding  map.

 In this paper we prove that Li-Yorke pairs of $A$-coupled-expanding system under some conditions have full Hausdorff  dimension on the invariant set. And we generalize the result of [5] which is on the  Hausdorff dimension of Li-Yorke pairs of dynamical systems topologically conjugate to the full shift, to the case of dynamical systems topologically semi-conjugate to some kinds of subshifts. Moreover we get a result on Hausdorff dimension of ``chaotic invariant set" for some kinds of A-coupled-expanding systems.
In this paper we investigate these topics by using the concept of symbolic geometric construction which was defined  in [6].

The rest of this paper is organized as follows. In Section 2 some basic concepts which will be used later are introduced. In section 3 we prove that  for some  $A$-coupled-expanding systems under some conditions, their  invariant Cantor sets  in which they are topologically conjugate to the subshift $\sigma _A$  become  limit sets of a symbolic geometric construction concerning the basic sets of the systems (Theorem 3.1). This means that Li-Yorke pairs of  $A$-coupled-expanding system under some conditions have full Hausdorff dimension on the invariant set(Remark 3.1). We also prove a theorem on the Hausdorff dimension of the set of Li-Yorke pairs for a strictly coupled-expanding map (Thorem 3.2).  In section 4, we generalize the result of [8] on the Hausdorff dimension of Li-Yorke pairs of dynamical systems topologically  conjugate to full shift and having a self-similar invariant set,  to the case of dynamical system topologically semi-conjugate to some kinds of subshifts (Theorem 4.1). And by using Theorem 4.1, we obtain a result on Hausdorff dimension of the set of Li-Yorke pairs of a strictly $A$-coupled-expanding system for some special matrices $A$ under some conditions(Theorem 4.2).

\section{Preliminaries}

In this section we  introduce main concepts  which are used in this paper. All the others concerned with topological and symbolic
dynamics are refered to the notations in [1].

\vskip0.5cm
\noindent{\bf Definition 2.1}{[6]} Let $(X, T)$ be a dynamical system on metric space $X$. A pair of points $(x, y)\in X^2$ is said to be  \emph {Li-Yorke pai}r for $T$ if 

\begin{displaymath}
\liminf\limits_{n\rightarrow \infty} d(T^n x, T^n y)=0 \quad \textrm {and} \quad  \limsup\limits_{n\rightarrow \infty} d(T^n x, T^n y)>0.
\end{displaymath}

Given an invariant set $\Lambda \subset X$, i.e. $f(\Lambda)=\Lambda$, we define the set of Li-Yorke pairs in $\Lambda$ for $T$ by 
\begin{displaymath}
LY_T(\Lambda)=\{(x,y) \in \Lambda ^2 | (x,y) \textrm{ is a Li-Yorke pair}\}.
\end{displaymath}

We say that Li-Yorke pairs in $\Lambda$ have full Hausdorff dimension for $T$ if the Hausdorff dimension of $LY_T(\Lambda)$ coincides with the Hausdorff dimension of $\Lambda ^2$, i.e.
\begin{displaymath}
\dim_H (LY_T(\Lambda))=\dim_H(\Lambda^2).
\end{displaymath}

\vskip0.5cm
\noindent{\bf Definition 2.2}{[11]}  Let $(X,d)$ be a metric space and $f:D\subset X \rightarrow X$. Let $A=((A)_{ij})_{m\times m}$ be a $m\times m$ transitive matrix, where $m\ge 2$.  
If there exist $m$  nonempty subsets $V_i(1\leq i\leq m)$ of $D$ with pairwise disjoint interiors such that 
\begin{displaymath}
f(V_i)\supset\bigcup\limits_{\substack{j\\(A)_{ij}=1}}V_j, \quad 1\leq i\leq m
\end{displaymath}
then $f$ is said to be \emph {$A$-coupled-expanding map}  (or system $(X, f)$ is said to be  \emph { $A$-coupled-expanding system}) in $V_i, 1\leq i\leq m$.\\

Further, the map $f$ is said to be \emph {strictly $A$-coupled-expanding map} (or system $(X, f)$ is said to be \emph {$A$-coupled-expanding system}) in $V_i, 1\leq i\leq m\quad$ if $\quad d(V_i, V_j)>0$ for all $1\leq i\neq j\leq m$.
In the special case that all entries of $A$ are 1s, the (strictly) $A$-coupled-expanding map is said to be  (strictly) coupled-expanding map.

\vskip0.5cm
\noindent{\bf Definition 2.3}[6]  Let $\Sigma _m^+=\{(i_1 \ldots i_n \ldots) : i_j =1, \ldots , m\}$.
 Let $Q\subset \Sigma _m^+$ be an invariant set of one sided full shift $\sigma$  on $\Sigma _m^+$  and $\{\Delta_{i_1 \ldots i_n}\}, (i_j =1, \ldots , m)$ be a family of closed sets in ${\mathbb R}^d$ called as \emph {basic sets} where $\{i_1 \ldots i_n\}$ is an admissible $n$-tuple with respect to $Q$, i.e., there exists $(j_1 \ldots j_n \ldots )\in Q$ such that $j_1=i_1, \ldots, j_n=i_n $ .

If for any admissible tuple  $(i_1 \ldots i_n i_{n+1})$ with repect to $Q$  it  is satisfied that \\

(i) $ \Delta _{i_1 \ldots i_n i_{n+1}} \subset  \Delta _{i_1 \ldots i_n }$,\\

(ii) $\Delta _{i_1 \ldots i_n }\cap \Delta _{j_1 \ldots j_n }=\emptyset$, $(i_1 \ldots i_n )\ne (j_1 \ldots j_n )$ \\

and \\
\begin{center}
$\lim \limits_{n\rightarrow \infty} \max \limits_{\begin{subarray}{l}(i_1 \ldots i_n ) \\ {admissible}\end {subarray}}  diam (\Delta_{i_1 \ldots i_n})=0$,\\
\end{center}
then we call the pair $(Q, \{\Delta _{i_1 \ldots i_n}\})$ \emph{symbolic geometric construction}. And the set

\begin{displaymath}
 F=\bigcap \limits_{n=1}^\infty \bigcup \limits_{\begin{subarray}{l}(i_1 \ldots i_n ) \\ {admissible}\end {subarray}} 
\Delta_{i_1 \ldots i_n} 
\end{displaymath}
is said to be \emph {limit set} of it.

This limit set is  Cantor-like set, i.e., it is perfect, nowhere dense and totally disconnected set.
The geometric construction $(Q, \{\Delta _{i_1 \ldots i_n}\})$ is said to be a \emph{simple geometric construction} if  $Q=\Sigma _m^+$,  and it is said to be \emph {Markov geometric construction} if   $Q=\Sigma _m^+(A)$ for a transtive matrix $A$(see [11] for  $\Sigma _m^+(A)$).

\vskip0.5cm

\noindent{\bf Definition 2.4}{[2]} Let  $S_i:{\mathbb R}^d \rightarrow {\mathbb R}^d (1\le i \le m)$ be a contraction map with contract ratio coefficient $c_i$, i.e., $|S_i(x)-S_i(y)|=c_i|x-y|,\quad 0 <c_i<1.$

If $K=\bigcup \limits_{i=1}^N S_i(K)$, then $K$ is said to be \emph{invariant set with respect to } $S=\{S_1,\ldots, S_N\}$.

If $K$ is invariant set with respect to $S=\{S_1,\ldots, S_N\}$ and for any $\alpha$ with  $\sum \limits_{i=1}^N c_i^\alpha =1 $, satisfies that 
\begin{displaymath}
H^\alpha (K)>0,\quad  H^\alpha (K_i \cap K_j)=0, \quad  (i\ne j)
\end{displaymath}
then $K$ is said to be  \emph {self-similar set}  where $H^\alpha$ is $\alpha$ dimension Hausdorff measure and $K_i=S_i(K).$

\section{$A$-coupled-expanding map with symbolic geometric construction and Hausdorff dimension of the set of Li-Yorke pairs for it.}

We now prove that for some  $A$-coupled-expanding maps under some conditions, their invariant Cantor sets in which they are topologically conjugate to the subshift $\sigma _A$  refer to the limit sets of symbolic geometric construction concerning the basic sets of the maps, so that their Li-Yorke pairs have full Hausdorff dimension. And we consider Hausdorff dimension of the set of Li-Yorke pairs for a strictly coupled-expanding map.

\begin{lemma}
Let $(X,d)$ be a metric space, $f:X\subset D \rightarrow X$ be a map and  $A$ be an $m \times m (m \ge 2)$ irreducible transitive matrix such that

\begin{displaymath}
 \quad \exists i_0 (1 \le i_0 \le m), \quad  \Sigma _{j=1}^m (A)_{i_0j}\ge 2. 
\end{displaymath}

Assume that there are $m$ compact subsets  $V_i(1\le i \le m) $  of $X$ with pairwise disjoint interiors such that $f$ is continuous and satisfies followings:\\

i) $f$ is a strictly $A$-coupled-expanding map  on the $V_i(1\le i \le m)$,\\

ii) there exist some constants $\lambda _1, \ldots, \lambda_m (\lambda_i>1)$ such that 

\begin{displaymath}
d(f(x), f(y))=\lambda_i d(x,y), \quad x, y \in V_i \quad  (1\le i \le m).
\end{displaymath}

Then $f$ has an invariant  Cantor set $V\subset \bigcup \limits_{i=1}^m V_i$ such that  $f$ in $V$ is topologically conjugate to the subshift $\sigma _A$.

\end {lemma}

\vskip0.5cm
\begin{proof} Put $\lambda _0=\min \limits_{1 \le i\le m}\lambda _i$, then for any $x,y \in V_i$,
 
\begin{displaymath}
d(f(x), f(y))\ge \lambda _0d(x,y).
\end{displaymath}

Thus the desired result follows immediately from  the  theorem 5.2 of [11]. 
\end{proof}

Next theorem shows that for some $A$-coupled maps, their invariant sets in which they are topologically conjugate to the subshift $\sigma _A$ have symbolic geometric construction.


\vskip0.5cm

\begin{theorem} Let $D\subset {\mathbb R}^d$ be a bounded and closed set and suppose that  $f:D \rightarrow {\mathbb R}^d$ satisfies all assumptions of above Lemma. Put $Q= \Sigma_m^+ (A)$ and for admissible sequence $(a_0 a_1 \ldots a_n)$ put 

\begin{displaymath}
 \Delta _{a_0 a_1 \ldots a_n}=\bigcap _{j=0}^n f^{-j}(V_{a_j}).
\end{displaymath}

Then a symbolic geometric construction with the family of basic sets $\{ \Delta _{a_0 a_1 \ldots a_n}\}$ is constructed and the limit set of this construction is coincide with the set  $V$ in the Lemma 3.1, i.e., 

\begin{displaymath}
\bigcap \limits_{n=0}^\infty \bigcup \limits_{\begin{subarray}{l}(a_0 a_1 \ldots a_n ) \\ {admissible}\end {subarray}} \Delta _{a_0 a_1 \ldots a_n} =V.
\end{displaymath}
In other words $f$ has an invariant Cantor set $V$, which becomes a limit set of Markov geometric construction $(\Sigma_m^+ , \{\Delta _{a_0 a_1 \ldots a_n}\})$, in which $f$ is topologically conjugate to the subshift $\sigma _A$.

\end{theorem}

\vskip0.5cm

\begin{proof} It is easy to see that for any $n\in \mathbb N$, $\Delta _{a_0 a_1 \ldots a_n a_{n+1}} \subset \Delta _{a_0 a_1 \ldots a_n}$ and 
\begin{displaymath}
f(\Delta _{a_0 a_1 \ldots a_n}) \subset \bigcap \limits_{j=0}^n f^{1-j}(V_{a_j}) \subset \bigcap \limits_{j=1}^n f^{1-j}(V_{a_j})=\Delta _{ a_1 \ldots a_n}.
\end{displaymath}
It follows  inductively that 
\begin{displaymath}
f^n(\Delta _{a_0 a_1 \ldots a_n}) \subset \Delta _{a_n}=V_{a_n}.
\end{displaymath}
On the other hand for any $x, y \in \Delta _{a_0 a_1 \ldots a_n} \subset V_{a_0}$, it follows that
\begin{displaymath}
d(x,y)=\frac{1}{\lambda_{a_0}} d(f(x), f(y))
\end{displaymath}
and 
\begin{displaymath}
f(x), f(y)\in \Delta_{a_1 \ldots a_n} \subset V_{a_1}, 
\end{displaymath}
 which means that

\begin{displaymath}
d(f(x), f(y))=\frac{1}{\lambda _{a_1}}d(f^2(x),(f^2(y)).
\end{displaymath}
Proceeding with these processes, we have

\begin{displaymath}
d(x,y)=\frac{1}{\lambda_{a_0}\lambda_{a_1} \ldots \lambda_{a_n}} d(f^n(x), f^n(y)) \le \frac {1}{(\min \lambda _i)^n} \textrm{diam} V_{a_n} \le \frac {1}{(\min \lambda_i)^n}\textrm{diam}D.
\end{displaymath}
It means that
\begin{displaymath}
\max \textrm{diam}\Delta_{a_0 a_1 \ldots a_n}\le \frac {1}{(\min \lambda_i)^n}\textrm{diam}D,
\end{displaymath}
therefore

\begin{displaymath}
\lim_{n\rightarrow \infty}\max\textrm{diam}\Delta_{a_0 a_1 \ldots a_n}=0.
\end{displaymath}

On the other hand since there is an invariant subset $V \subset \bigcup\limits_ {i=1}^m V_i$ in which  $f$  is topologically conjugate to the subshift $\sigma _A$ by Lemma 3.1, from Theory 4.1 in [11], $\bigcap \limits_{n=0}^\infty f^{-n}(V_{a_n})$ is singleton for any $\alpha =(a_0a_1\ldots a_n\ldots) \in \Sigma_m^+(A)$ and 

\begin{displaymath}
V=\bigcup \limits_{\alpha \in \Sigma_m^+(A)}\bigcap \limits_{n=0}^\infty f^{-n}(V_{a_i}).
\end{displaymath}
Obviously 

\begin{displaymath}
\bigcup \limits_{\alpha \in \Sigma_m^+(A)}\bigcap \limits_{n=0}^\infty f^{-n}(V_{a_i})=\bigcap \limits_{n=0}^\infty \bigcup \limits_{\begin{subarray}{l}(i_1 \ldots i_n ) \\ {admissible}\end {subarray}} \Delta _{a_0 a_1 \ldots a_n}.
\end{displaymath}
The proof is thus complete.
\end{proof}


{\sc Remark 3.1} From Theorem 5.1 in [5] and above Theorem 3.1, it follows that  Li-Yorks pairs of coupled-expanding map satisfying the conditions of Lemma 3.1 have  full Hausdorff dimension in the invariant set $V$.

\vskip0.5cm
Next Theorem concerns on  the Hausdorff dimension of this invariant set $V$ , so-called ``chaotic set", for some kinds of coupled-expanding dynamical systems.


\begin{theorem} Let $D\subset {\mathbb R}^n$ be a closed bounded set and $f:D\rightarrow {\mathbb R}^n$ be a strictly coupled-expanding map in m disjoint compact subsets $V_i \subset D^ d (1\le i\le m)$ and continuous in $\bigcup \limits_{i=1}^m V_i$.  If there are some constants $\lambda _1, \ldots , \lambda_m (\lambda_i >1)$ such that 

\begin{displaymath}
d(f(x), f(y))=\lambda_i d(x,y),\space  \quad x,y \in V_i  \quad (1 \le i\le m),
\end{displaymath}
then the Hausdorff dimension of the limit set $V$ for the symbolic geometric construction $(Q=\Sigma _m^+,  \quad  \Delta _{a_0 a_1 \ldots a_n}=\bigcap _{j=0}^n f^{-j}(V_{a_j}))$ is the solution of the equation

\begin{displaymath}
(\frac {1}{\lambda_1})^p+\ldots +(\frac {1}{\lambda_m})^p=1.
\end{displaymath}
And we have
\begin{displaymath}
\dim_H LY_f(V)=2p_0
\end{displaymath}
where  $p_0$ is the solution of this equation.
\end{theorem}

\vskip0.5cm

\begin{proof}
For any $i\in {1,2, \ldots , m}$, put $U_i=\{\alpha=(a_0a_1\ldots) \in \Sigma_m^+:a_0=i\}$.
From the assumption for $f$,  for any $\alpha=(a_0a_1\ldots) \in \Sigma_m^+$ the set $\bigcap _{n=0}^\infty  f^{-n}(V_{a_n})$ is a  singleton.
Now define a map $g:\Sigma_m^+ \rightarrow V$ as follows:

\begin{displaymath}
\alpha=(a_oa_1\ldots ) \mapsto \bigcap _{n=0}^\infty  f^{-n}(V_{a_n}).
\end{displaymath}

Then, from Theorem 4.1 in [11] $g$ is homeomorphism and we have $f\circ g=g \circ \sigma$.
Obviously we have $g(U_i)\subset V_i$ and it follows that 

\begin{displaymath}
f(g(U_i))=g(\sigma (U_i))=g(\Sigma_m^+)=V.
\end{displaymath}
This means that $V$ can be formed by expanding of $\lambda _i$ times of $g(U_i)$. Therefore for any $i(1\le i\le m)$ putting $S_i=(f|_{V_i})^{-1}\big|_V$, then $S_i$ is a contraction map with contract  ratio coefficient $\frac{1}{\lambda_i}$. In fact, for any $x,y \in V$ there are $t, s \in V_i$ such that $f(t)=x, f(s)=y$ since $f$ is expanding in $V_i (1\le i \le m)$. Therefore,

\begin{displaymath}
 \begin{array}{ll}
S_i(x)=(f|_{V_i})^{-1}\circ f(t)=t,\\
S_i(y)=(f|_{V_i})^{-1}\circ f(s)=s.
\end{array}
\end{displaymath}
Thus we have 
\begin{displaymath}
d(S_i(x), S_i(y))=d(t,s)=\frac {1}{\lambda_i}d(x,y).
\end{displaymath}

On the other hand if $i\neq j$, then $S_i(V)\cap S_j(V)=\emptyset$ since $S_i(V)=g(U_i) \subset V_i$. Hence $V$ is a self-similar set. Note that

\begin{displaymath}
V=\bigcup \limits_{i=1}^m g(U_i)=\bigcup \limits_{i=1}^m S_i(V).
\end{displaymath}
Therefore, by the Theorem 2 in [4] $\dim _H (V)$ is equal to the solution of the equation

\begin{displaymath}
(\frac {1}{\lambda_1})^p+\ldots +(\frac {1}{\lambda_m})^p=1.
\end{displaymath}

Then, since $f$ satisfies in $V$ the condition of Theorem 5.1 in [5], it follows that  

\begin{displaymath} 
\dim_H(LY_f(V))=\dim_H(V\times V)=2\dim_HV=2p_0
\end{displaymath}
where $p_0$ is a solution of above mentioned equation. Thus the proof is complete.
\end{proof}

\section{Li-Yorke pairs of full dimension for systems topologically semi-conjugate to a subshift}

In this section we  generalize the result of [5] on the Hausdorff dimension of Li-Yorke pairs of dynamical systems which are topologically conjugate to a full shift and have a self-similar invariant set,  to the case of dynamical system topologically semi-conjugate to some kinds of subshifts. Moreover we consider Hausdorff dimension of ``chaotic invariant set"  for the systems.

We  consider a kind of matrices as following:
\begin{displaymath}
\label{1}
A=\left(\begin{array}{cccccc}
		&& 1			\\
		&& \vdots 			\\
	 1	& \cdots 	&0	& \cdots 	& 1\\
		&& \vdots 			\\
		&& 1			\\
\end{array}\right)
\end{displaymath}
where all the entries of  $i$ th row and  $i$ th column are equal to 1s except that  $(A)_{ii}=0$, while other entries may be arbitrary.

We are going to prove that the result of [5] above mentioned holds as well for the map topologically semi-conjugate to the subshift $\sigma_A$ for this kind of matrices $A$.

First, we consider the matrix $A$  which has  $i$ th row and  $i$ th column consist of 1s while other entries are all 0. Especially, to simplify our consideration we are going to fix the matrix $A$ as following:

\begin{displaymath}
A=\left(\begin{array}{cccc}
	0     &   1   &  \cdots &   1        \\
	1     &   0    & \cdots &    0    \\
	\vdots & \vdots & \ddots & \vdots \\
	1      &  0     & \cdots &    0
\end{array}\right) ,
\end{displaymath}
the other cases can be treated similar to this.
We can see that $\Sigma _m^+(A)$ and $\Sigma _m^+$ are homeomorphic. In fact, for any $s\in \Sigma _m^+(A)$, assume $\bar{s}$ is the sequence obtained from $s$ by eliminating one digit 1  which lies behind of  elements different from 1 in $s$, and define a map $\Phi: \Sigma _m^+(A)\rightarrow \Sigma_m^+$,  by $\Phi(s)=\bar{s}$, then we can see easily that $\Phi$ is homeomorphism.

Now let $\Lambda \subset {\mathbb R}^d$ be a self-similar set constructed by a family of contracting maps $S=\{S_1, \ldots, S_m\}$ satisfying :\\

1) the contract ratio coefficient of $S_i$ is $c_i$,\\

2) there exists a compact set $K\subset {\mathbb R}^d$ such that $S_i(K)\subset K$  for any $i (1\le i\le m)$ and $S_i(K)\cap S_j(K)=\emptyset$ if $i\neq j$,\\

3) $\Lambda =\bigcup\limits_{i=1}^m S_i(\Lambda)$.

And define a map $\pi: \Sigma _m^+\rightarrow \Lambda$ by 

 \begin{displaymath}
\pi(\alpha)=\lim\limits_{n\rightarrow \infty}S_{a_n}\circ \ldots \circ S_{a_0}(K),\quad  \alpha =(a_0a_1\ldots)\in \Sigma_m^+.
\end{displaymath}

It is easy to see that $\pi$ is homeomorphism and  a map $\pi_A:\Sigma _m^+(A)\rightarrow \Lambda$ defined by $\pi_A=\pi \circ \Phi$ is obviously  homeomorphism.

Then we can get following Lemma.

\begin{lemma} Let $f:{\mathbb R}^d \rightarrow {\mathbb R}^d$ be a map with a compact invariant set $\Lambda$. If $(\Lambda, f)$ topologically conjugate to one-sided subshift $(\Sigma _m^+(A), \sigma_A)$ with $f\circ \pi_A=\pi_A \circ \sigma_A$,  and $\Lambda$ is self-similar, then the Li-Yorke pairs  for $f$ have full Hausdorff dimension, i.e., 
 \begin{displaymath}
 \dim_H(LY_f(\Lambda))=\dim_H \Lambda\times \Lambda
\end{displaymath}
where the first row and the first column of $A$ consist of 1s and other entries are all 0 .
\end{lemma}

\begin{proof}

Let $s\in \Sigma _m^+(A)$ and ${\cal N} =(N_n)$ be a sequence of natural numbers. Consider the set 
 \begin{displaymath} \begin{array}{l}
\Sigma _{\cal N}^A(S):=\{t\in \Sigma _m^+(A)\big | t_k=s_k, k\in \{u_i, u_i+i\};\\ \quad  t_{u_i+i+1}=1,  t_{u_i+i+2}=(1+s_{u_i+i+2})(\mod m),   t_{u_i+i+3}=1, t_{u_{i+1}-1}=1, i=0, 1, \ldots\},
\end{array}
\end{displaymath}
where $u_0=0$, $u_1=N_0+5$  and $u_i$ is given by the recursion $ u_{i+1}=u_i+N_i+i+6$.

Then an element $t\in \Sigma _{\cal N}^A(S)$ has the form 
 \begin{displaymath}
t=s_0 1\tilde{t}_2 1\underbrace{ \cdots }_{N_0}1s_{u_1}s_{u_1+1}1\tilde{t}_{u_1+3}1\underbrace{ \cdots }_{N_1} 1s_{u_2}s_{u_2+1}s_{u_2+2}1\tilde{t}_{u_2+4}1 \underbrace{ \cdots }_{N_2}1 \cdots ,
\end{displaymath}
where 
 \begin{displaymath}
\begin{array}{ll}
\tilde{t}_2=(1+s_2)\mod m,\\
\tilde{t}_{u_1+3}=(1+s_{u_1+3})\mod m,\\
\tilde{t}_{u_2+4}=(1+s_{u_2+4})\mod m ,\\
\ldots \quad .
\end{array}
\end{displaymath}

Then a pair $(s,t) \in \Sigma _m^+\times \Sigma _m^+(A), t\in \Sigma _{\cal N}^A(s)$,  is a Li-Yorke pair for $\sigma_A$. In fact it holds that

 \begin{displaymath}
\lim\limits_{i\rightarrow \infty}d(\sigma_A^{u_i}(s), \sigma_A^{u_i}(t))\le \lim\limits_{i\rightarrow \infty} 2^{-i}=0,
\end{displaymath}
 \begin{displaymath}
\lim\limits_{i\rightarrow \infty}d(\sigma_A^{u_i+i+1}(s), \sigma_A^{u_i+i+1}(t))\ge \frac{1}{2}>0.
\end{displaymath}

Now define the map $pr: \Sigma _m^+(A)\rightarrow \Sigma _{\cal N}^+(A)$ by 

 \begin{displaymath}
pr(t)=s_0 1\tilde{t}_21 t_0\ldots t_{N_0-1} 1s_{u_1}s_{u_1+1}1\tilde{t}_{u_1+2}1t_{N_0}\ldots t_{N_0+N_1-1}1\ldots \in  \Sigma _{\cal N}^+(A) 
\end{displaymath}
for $ t=(t_0t_1\ldots )\in \Sigma _m^+(A)$. We can see that the map $pr$ is a continuous  injection.

Let  

 \begin{displaymath}
\bar{\Sigma}=pr(\Sigma_m^+(A))\subset \Sigma _{\cal N}^A(s),
\end{displaymath}
 \begin{displaymath}
\Lambda_{\cal{N}}(s)=\pi_A(\Sigma _{\cal N}^A(s)),
\end{displaymath}
 \begin{displaymath}
\tilde{\Lambda}=\pi_A(\bar{\Sigma}).
\end{displaymath}
Then $\tilde{\Lambda}\subset \Lambda_{\cal{N}}(s)$.

Now we prove that

 \begin{displaymath}
\dim_H \Lambda_{\cal{N}}(s)=\dim_H \Lambda=D 
 \end{displaymath}
where  $D$ is a solution of the equation 
 \begin{displaymath}
c_1^D+\ldots +c_m^D=1
 \end{displaymath}
 for a sequence $(N_n)$ satisfying 
 \begin{displaymath}
\lim\limits_{M\rightarrow \infty} \frac{(M+6)^2}{\sum\limits_{n=0}^{M-1}N_n}=0
\end{displaymath}
(for example $N_n=n^2$). 

It is clear that $\dim_H\Lambda =D$ from the Theorem 2 in [10]. Since $\Lambda_{\cal{N}}(s) \subset \Lambda$, this yields  $\dim_H \Lambda_{\cal{N}}(s) \le D$.

For the opposite inequality it is sufficient to prove that $\dim _H(\tilde{\Lambda}) \ge D$ since $\Lambda_{\cal {N}} (s )\supset \tilde{\Lambda}$.
Now let $\nu$ be a Bernoulli measure on $\Sigma _m^+$ corresponding to the probability vector $(c_1^D, \ldots, c_m^D)$, 
and define 

 \begin{displaymath}
\mu=\nu\circ \Phi \circ {pr}^{-1} \circ {\pi_A}^{-1}.
 \end{displaymath}
Then $\mu$ becomes a probability measure on the $\tilde{\Lambda}$.

If we prove that 

 \begin{displaymath}
\liminf\limits_{\rho \rightarrow 0}\frac{\log \mu(B_\rho (x))}{\log \rho}\ge D
 \end{displaymath}
for any $x\in \tilde{\Lambda}$, then we shall  have $\dim _H(\tilde {\Lambda})\ge D$ using the Theorem 6.6.3 in [7].

By bijectivity of the map $pr$ on $\bar{\Sigma}$, for any  $x\in \tilde{\Lambda}$ there is a unique sequence $\alpha =(a_0a_1\ldots) \in \Sigma _m^+ $ such that $\pi_A(pr(\Phi^{-1}(\alpha)))=x$. Since $\{\pi_A \circ pr \circ \Phi^{-1}([a_0, \ldots, a_k])\}_{k=0}^\infty$ is a  contracting family  of compact subsets containing $x$, of which diameter goes  to 0, we have 

 \begin{displaymath}
\{x\}=\bigcap \limits_{k=0}^\infty \pi_A \circ pr \circ \Phi^{-1}([a_0, \ldots, a_k]),
\end{displaymath}
where 

 \begin{displaymath}
[a_0, \ldots, a_k]=\{t=(t_i)\in \Sigma_m^+|\quad t_i=a_i \quad  \textrm{for} \quad   i=1, \ldots, k \} .
\end{displaymath}
And it is clear that $ pr \circ \Phi^{-1}([a_0, \ldots, a_k])$ is also cylinder set in $\bar
{\Sigma}$.

For a cylinder $[b_0, \ldots, b_k]$, denote $c([b_0, \ldots, b_k])$ as follows:

\begin{displaymath}
c([b_0, \ldots, b_k])=c'_{b_i}\ldots c'_{b_k},
\end{displaymath}
\begin{displaymath}
c'_i=\left\{ \begin{array}{ll} \frac{c_i}{c_1} & i \neq 1,\\ c_1 & i=1.\end{array} \right.
\end{displaymath}

Then  $c([b_0, \ldots, b_k])$ is actually a product of $c_i$s because of the property of $A$. Put 

\begin{displaymath}
d=\min \limits_{i \neq j} \textrm{dist}(S_i(K),(S_j(K) )
\end{displaymath}
where
\begin{displaymath}
\textrm{dist}(A,B)=\inf\{\textrm{dist}(x,y)|x\in A, y\in B \}.
\end{displaymath}

Then we can see easily that the sequence $\{d\cdot c(pr \circ \Phi^{-1}([a_0, \ldots, a_k]))\}_{k=1}^\infty$ converges to $0$ as $k\rightarrow \infty$.
Therefore, for any $\rho>0$ there is a $k=k(\rho)$ such that 
\begin{displaymath}
d\cdot c(pr \circ \Phi ^{-1}([a_0, \ldots , a_k]))\le \rho <d \cdot c(pr \circ \Phi ^{-1}([a_0, \ldots , a_{k-1}])).
\end{displaymath}

Assume that $(\bar {a_0}, \ldots,\bar {a_k} ) \neq (a_0, \ldots, a_k)$, i.e., there is an $l\in \{0, \ldots, k\}$ such that $\bar{a_i}=a_i (i=0, \ldots, l-1), \bar{a_l}\neq a_l$. Then we have 

\begin{displaymath}
\textrm{dist}\left(\pi_A \circ pr \circ \Phi^{-1}([a_0, \ldots, a_k]), \{\pi_A \circ pr \circ \Phi^{-1}([\bar{a_0}, \ldots, \bar{a_k}])\right) \ge
\end{displaymath}
\begin{displaymath}
 \ge \textrm{dist} \left(\pi_A \circ pr \circ \Phi^{-1}([a_0, \ldots, a_l]), \{\pi_A \circ pr \circ \Phi^{-1}([\bar{a_0}, \ldots, \bar{a_l}])\right).
\end{displaymath}

We can denote $pr \circ \Phi^{-1}([a_0, \ldots, a_{l-1}])$ by $ [u_0, \ldots, u_t]$ since it is a cylinder in $\bar{\Sigma}$.
And put 
\begin{displaymath}
S_{[i_1, \ldots, i_k]}=S_{i_k}\circ  \ldots \circ S_{i_1}.
\end{displaymath}

From the definition of $\pi_A$ we have 

\begin{displaymath}
\pi_A \circ pr \circ \Phi^{-1}([a_0, \ldots, a_{l-1}])\subset S_{\Phi([u_0, \ldots, u_t])}(K),
\end{displaymath}

\begin{displaymath}
\pi_A \circ pr \circ \Phi^{-1}([a_0, \ldots, a_l])\subset S_{\Phi([u_0, \ldots, u_t, a_l])}(K) \subset S_{\Phi([u_0, \ldots, u_t])}(K),
\end{displaymath}

\begin{displaymath}
\pi_A \circ pr \circ \Phi^{-1}([\bar{a_0}, \ldots, \bar{a_l}])\subset S_{\Phi([u_0, \ldots, u_t, \bar{a_l}])}(K) \subset S_{\Phi([u_0, \ldots, u_t])}(K).
\end{displaymath}

Therefore,

\begin{displaymath}
\textrm{dist}\left(\pi_A \circ pr \circ \Phi^{-1}([a_0, \ldots, a_l]), \{\pi_A \circ pr \circ \Phi^{-1}([\bar{a_0}, \ldots, \bar{a_l}])\right) \ge
\end{displaymath}

\begin{displaymath}
\ge \textrm{dist}\left(S_{\Phi([u_0, \ldots, u_t, a_l])}(K), S_{\Phi([u_0, \ldots, u_t, \bar{a_l}])}(K)\right)
\end{displaymath}

\begin{displaymath}
\ge d\cdot c([u_0, \ldots, u_t])=d\cdot c( pr \circ \Phi^{-1}([a_0, \ldots, a_{l-1}]))>\rho.
\end{displaymath}

This means that 
\begin{displaymath}
\textrm{dist}\left(\pi_A \circ pr \circ \Phi^{-1}([a_0, \ldots, a_k]), \pi_A \circ pr \circ \Phi^{-1}([\bar{a_0}, \ldots, \bar{a_k}])\right)>\rho,
\end{displaymath}
thus 
\begin{displaymath}
\tilde{\Lambda} \cap B_\rho (x)\subset \pi_A \circ pr \circ \Phi^{-1}([a_0, \ldots, a_k])
\end{displaymath}
and 

\begin{displaymath}
\mu (B_\rho (x))\le \mu(\pi_A \circ pr \circ \Phi^{-1}([a_0, \ldots, a_k]))=\nu ([a_0, \ldots, a_k])=(c_{a_0}\cdot \ldots \cdot c_{a_k})^D.
\end{displaymath}

On the other hand,
\begin{displaymath}
c( \Phi^{-1}([a_0, \ldots, a_k]))=c_{a_0}\cdot \ldots \cdot c_{a_k}.
\end{displaymath}
In fact, $\Phi ^{-1}(a_0a_1\ldots)$ is the sequence obtained by setting 1 behind of  each digit of $(a_0a_1\ldots)$ not being 1, and  denoting the digits not being 1 by $a_{n_1}, \ldots, a_{n_p}$ in order, we have

\begin{displaymath}
c( \Phi^{-1}([a_0, \ldots, a_k]))=c_{a_0}\cdot \ldots c_{a_{n_1-1}}\cdot \frac{c_{a_{n_1}}}{c_1}\cdot c_1 \cdot \ldots \cdot c_{a_{n_2-1}}\cdot \frac{c_{a_{n_2}}}{c_1}\cdot c_1 \cdot \ldots \cdot c_{a_{n_p-1}}\cdot \frac{c_{a_{n_p}}}{c_1}\cdot c_1\cdot \ldots \cdot c_{a_k}
\end{displaymath}

\begin{displaymath}
=c_{a_0}\cdot \ldots \cdot c_{a_k}.
\end{displaymath}
Therefore, 
\begin{displaymath}
\mu (B_\rho (x))\le (c( \Phi^{-1}([a_0, \ldots, a_k])))^D=
\end{displaymath}

\begin{displaymath}
=\frac{(c( \Phi^{-1}([a_0, \ldots, a_k])))^D}{(c(pr \circ  \Phi^{-1}([a_0, \ldots, a_k])))^D}\cdot(c(pr \circ  \Phi^{-1}([a_0, \ldots, a_k])))^D 
\end{displaymath}

\begin{displaymath}
\le \frac{(c( \Phi^{-1}([a_0, \ldots, a_k])))^D}{(c(pr \circ  \Phi^{-1}([a_0, \ldots, a_k])))^D} \cdot d^{-D}\cdot \rho^D.
\end{displaymath}
It follows that

\begin{displaymath}
\log \mu (B_\rho (x))\le D\big(\log \rho-\log d-\log \frac {c(pr \circ  \Phi^{-1}([a_0, \ldots, a_k]))}{c( \Phi^{-1}([a_0, \ldots, a_k]))}\big)
\end{displaymath}

\begin{displaymath}
\le D\big(\log \rho-\log d-\delta (k) \cdot \log \underline{c}\big),
\end{displaymath}
where $\underline{c}=\min \limits_i c_i$ and 

\begin{displaymath}
\delta (k)=\sharp (pr \circ  \Phi^{-1}([a_0, \ldots, a_k]))-\sharp ( \Phi^{-1}([a_0, \ldots, a_k])),
\end{displaymath}
here $\sharp$ denotes the length of the cylinder.
Then, dividing  both sides of  above inequality by $\log \rho$  we have 

\begin{displaymath}
\frac{\log \mu (B_\rho (x))}{\log \rho}\ge D+D\big(- \frac{\log d}{\log \rho}-\delta (k)\cdot \frac{\log \underline{c}}{\log \rho}\big) \ge
\end{displaymath}

\begin{displaymath}
\ge  D+D\big(- \frac{\log d}{\log \rho}-\frac{\delta (k)\cdot \log \underline{c}}{\log d+\log c(pr \circ \Phi^{-1}([a_0, \ldots, a_{k-1}]) )} \big)
\end{displaymath}

\begin{displaymath}
\ge D+D \big(-\frac{\log d}{\log \rho}-\frac{\delta (k)\cdot \log \underline{c}}{\log d+(k+\delta (k-1)) \cdot \log \overline{c}} \big)
\end{displaymath}
where $\overline{c}=\max \limits_i c_i$. 

Since $\lim \limits_{\rho \rightarrow 0} k(\rho)=\infty$, now we prove that $\lim \limits_{k\rightarrow \infty }\frac{\delta (k)}{k}=0$.
For $k\in \mathbb N$ there is a $M=M(k)$ such that 
\begin{displaymath}
\sum \limits_{n=0}^{M-1} N_n< 2(k+1) \le \sum \limits_{n=0}^M N_n
\end{displaymath}
where we note $\sharp ( \Phi^{-1}([a_0, \ldots, a_k])) $ is no more than $2(k+1)$.  From the definition of $\delta, \Sigma _{\cal N}^A(s)$ and $ pr$, we have 

\begin{displaymath}
\delta (k)<\sum \limits_{n=0}^M (n+6) <(M+6)^2 ,
\end{displaymath}

\begin{displaymath}
\frac{\delta (k)}{2(k+1)}<\frac{(M+6)^2}{\sum \limits_{n=0}^{M-1}N_n}\rightarrow 0 (k\rightarrow \infty).
\end{displaymath}
Therefore 

\begin{displaymath}
\frac{\delta (k)}{k}\rightarrow 0 (k\rightarrow \infty).
\end{displaymath}
Thus we have

\begin{displaymath}
\liminf \limits_{\rho \rightarrow 0} \frac {\log \mu (B_\rho (x))}{\log \rho}\ge D.
\end{displaymath}

Now put 
\begin{displaymath}
\Pi_{\cal N}=\{(s,t)| s\in \Sigma_m^+(A), t\in \Sigma _{\cal N}^A(s)\}.
\end{displaymath}
Since $(s,t) \in \Pi_{\cal N}$ is Li-Yorke pair for $\sigma_A$, we have 
\begin{displaymath}
\Pi_{\cal N} \subset \textrm{LY}_{\sigma_A}\big(\Sigma_m^+(A)\big).
\end{displaymath}

Let
\begin{displaymath}
S^A=\{(x,y)\in \Lambda \times \Lambda |x\in \Lambda, y\in \Lambda_{\cal N}(\pi_A^{-1}(x))\}=\pi _A(\Pi_{\cal N}),
\end{displaymath}
then we have

\begin{displaymath}
S^A\subset \textrm{LY}_f (\Lambda).
\end{displaymath}
In fact, if $(s,t)\in \textrm{LY}_{\sigma_A}\big(\Sigma_m^+(A)\big)$, then from its definition we have 
\begin{displaymath}
\liminf \limits_{n\rightarrow \infty}d(\sigma _A^n(s), \sigma _A^n(t))=0, \quad  \limsup \limits_{n\rightarrow \infty}d(\sigma _A^n(s), \sigma _A^n(t))>0.
\end{displaymath}
Since $\pi_A$ is continuous, it follows that
\begin{displaymath}
\liminf \limits_{n\rightarrow \infty}d(\pi_A(\sigma _A^n(s)),\pi_A( \sigma _A^n(t)))=0, \quad  \limsup \limits_{n\rightarrow \infty}d(\pi_A( \sigma _A^n(s), \pi_A( \sigma _A^n(t)))>0,
\end{displaymath}
and since $f\circ \pi_A=\pi_A \circ \sigma_A$, we get 
\begin{displaymath}
\liminf\limits_{n\rightarrow \infty}d(f^n(\pi_A(s)), f^n(\pi_A(t)))=0, \quad  \limsup\limits_{n\rightarrow \infty}d(f^n(\pi_A(s)), f^n(\pi_A(t)))>0.
\end{displaymath}
Hence,
\begin{displaymath}
(\pi_A(s), \pi_A(t)) \in \textrm{LY}_f(\Lambda).
\end{displaymath}
Now using $\Pi_{\cal N} \subset \textrm{LY}_{\sigma_A}\big(\Sigma_m^+(A)\big)$ we get 
\begin{displaymath}
S_A=\pi_A(\Pi_{\cal N})\subset  \textrm{LY}_f(\Lambda).
\end{displaymath}

On the other hand  the Theorem 4.1 in [8]  implies that 
\begin{displaymath}
\dim_HS^A=\dim_H \Lambda \times \Lambda ,
\end{displaymath}
and using $S^A\subset  \textrm{LY}_f(\Lambda)\subset \Lambda \times \Lambda$ we can get 
\begin{displaymath}
\dim_H \textrm{LY}_f(\Lambda)=\dim_H \Lambda \times \Lambda.
\end{displaymath}

\end{proof}

Next, Let $A$ be a transitive matrix such that all the entries of  $i$th row and $i$th column are 1s. As above, we assume that $i$ = 1 . Now, we will generalize definition of the map $\Phi: \Sigma _m^+(A)\rightarrow \Sigma_m^+$. For any $s\in \Sigma _m^+(A)$, assume $\bar{s}$ is the sequence obtained from $s$ by eliminating one digit  which lies behind of  elements different from 1 in $s$, and define a map $\Phi: \Sigma _m^+(A)\rightarrow \Sigma_m^+$,  by $\Phi(s)=\bar{s}$, then we can prove that $\Phi$ is continuous surjection. (Note that eliminating digit might not be 1.) \\
\indent Similarly to above consideration, we define a map $\pi_A:\Sigma _m^+(A)\rightarrow \Lambda$ as $\pi_A=\pi \circ \Phi$ (the map $\pi:\Sigma_m^+\rightarrow\Lambda$ is already defined as $\pi(\alpha)=\lim\limits_{n\rightarrow}S_{a_n}\circ \ldots \circ S_{a_0}(K),  \alpha =(a_0a_1\ldots)\in \Sigma_m^+$) and then $\pi_A$ is obviously continuous surjection. 

\vskip0.5cm

Now, by using Lemma 4.1 we can generalize Theorem 5.1 in [5] to the case of the map topologically semi-conjugate to some kinds of subshifts .

\begin{theorem} Let  $A$ be a transitive matrix such that all the entries of  $i$th row and $i$th column are 1s for at least an $i$ ($1\le i \le m$). Let $f:{\mathbb R}^d \rightarrow {\mathbb R}^d$ be a map with the invariant set $\Lambda$ which is a self-similar compact set. If $(\Lambda, f)$ is topologically semi-conjugate to an one-sided subshift $(\Sigma _m^+(A), \sigma_A)$ with $f\circ \pi_A=\pi_A \circ \sigma_A$,  then the Li-Yorke pairs have full Hausdorff dimension for $f$, i.e., 
 \begin{displaymath}
 \dim_H(LY_f(\Lambda))=\dim_H \Lambda\times \Lambda
\end{displaymath}
\end{theorem}

\begin{proof} Without losing generality, we assume that $i$ = 1. Put
\begin{displaymath}
A'=\left(\begin{array}{cccc}
	0     &   1   &  \cdots &   1        \\
	1     &   0    & \cdots &    0    \\
	\vdots & \vdots & \ddots & \vdots \\
	1      &  0     & \cdots &    0
\end{array}\right).
\end{displaymath}
Then $\Sigma_m^+(A')\subset\Sigma_m^+(A)$, and the restriction of $\Phi$ to $\Sigma_m^+(A')$ is homeomorphic. Thus
\begin{displaymath}
\pi_A|_{\Sigma_m^+(A')}=\pi \circ \Phi|_{\Sigma_m^+(A')}=\pi_{A'}
\end{displaymath} 
is also homemorphism from $\Sigma_m^+(A')$ to $\Lambda$. Also
\begin{displaymath}
f\circ \pi_{A'}=f\circ \pi_A|_{\Sigma_m^+(A')}=\pi_A \circ \sigma_A|_{\Sigma_m^+(A')}=\pi_{A'} \circ \sigma_{A'}
\end{displaymath} 
and therefore $(\Lambda, f)$ is topologically conjugate to  $(\Sigma _m^+(A'), \sigma_{A'})$, i.e,
\begin{displaymath}
f\circ \pi_{A'}=\pi_{A'} \circ \sigma_{A'}.
\end{displaymath} 

From Lemma 4.1, we have
\begin{displaymath}
 \dim_H(LY_f(\Lambda))=\dim_H \Lambda\times \Lambda.
\end{displaymath}
\end{proof}

\vskip0.5cm

Next theorem concerns on the Hausdorff dimension of ``chaotic invariant set"  for  $A$-coupled-expanding systems for special matrix $A$.

\begin{theorem} Let  $A$ be $m\times m$ transitive matrix such that all the entries of  first row and first column are 1s  while other entries are all 0. Assume that there are $m$ disjoint compact subsets  $V_i(1\le i \le m) (m \ge 2)$  of $X$ such that $f$ satisfies the conditions in the Lemma 3.1, i.e., $f$ is continuous and satisfies followings:\\

i) $f$ is a strictly $A$-coupled-expanding map  on the $V_i(1\le i \le m)$,\\

ii) there exist some constants $\lambda _1, \ldots, \lambda_m (\lambda_i>1)$ such that 

\begin{displaymath}
d(f(x), f(y))=\lambda_i d(x,y),  \quad x, y \in V_i (1\le i \le m).
\end{displaymath}

Then the Hausdorff dimension of the Cantor invariant subset $V\subset \bigcup \limits_{i=1}^m V_i$ in which f is topologically conjugate to subshift $\sigma_A$ (see lemma 3.1), is the solution  of the equation 

\begin{displaymath}
(\frac{1}{\lambda_1})^p+(\frac{1}{\lambda_1 \lambda_2})^p+\ldots +(\frac{1}{\lambda_1 \lambda_m})^p=1.
\end{displaymath}
Moreover 

\begin{displaymath}
\dim_H\textrm{LY}_f(V)=2p_0,
\end{displaymath}
where   $p_0$ is the solution of this equation.
\end{theorem}

\begin{proof} Put
\begin{displaymath}
U_i=\{\alpha \in \Sigma_m^+(A):a_0=i \}.
\end{displaymath}
Then we have 
\begin{displaymath}
\sigma_A(U_1)=\Sigma_m^+(A), \quad \sigma_A(U_i)=U_1  (2\le i\le m).
\end{displaymath}
Using the same way as in the proof of Theorem 3.2, there exists a homeomorphism $g:\Sigma_m^+(A) \rightarrow V$ such that 
\begin{displaymath}
f\circ g=g\circ \sigma_A, \quad g(U_i) \subset V_i (i=1, \ldots , m)
\end{displaymath}
and 
\begin{displaymath}
f(g(U_1))=g(\sigma_A(U_1))=V.
\end{displaymath}
Therefore we can see that $V$ is obtained by $ \lambda _1$ times expanding of $g(U_1)$.
Since 
\begin{displaymath}
f^2(g(U_i))=f(g(\sigma_A(U_i)))=f(g(U_1))=V \quad (2\le i\le m),
\end{displaymath}
we can also see that $V$ is obtained by $ \lambda _1  \lambda _i$ times expanding of $g(U_i)$.
This leads to the fact that by  putting

\begin{displaymath}
S_1=(f|_{V_1})^{-1}\big|_V, \quad S_i=(f^2|_{V_i})^{-1}\big|_V,
\end{displaymath}
$S_1, S_i$ are  contracting maps with the contract ratio coefficients $\frac {1}{\lambda _1}, \frac {1}{\lambda_1 \lambda_i}$ and $V$ is the self-similar set for $\{S_1, \ldots, S_m\}$.

Thus, from  Theorem 2 in [10], $\dim_H V$ is the solution of the equation 
\begin{displaymath}
(\frac{1}{\lambda_1})^p+(\frac{1}{\lambda_1 \lambda_2})^p+\ldots +(\frac{1}{\lambda_1 \lambda_m})^p=1.
\end{displaymath}

And by using the Theorem 4.1, we have
\begin{displaymath}
\dim_H \textrm{LY}_f(V)=\dim_H V\times V=2\dim_H V=2p_0.
\end{displaymath}

\end{proof}

\section{Conclusion}

Through this work we have several interesting observations: ``Chaotic invariant set" for some kind of $A$-coupled-expanding maps refers to a limit set of symbolic geometric construction concerning the basic sets of them(in this paper, by ``chaotic invariant set" we mean the invariant Cantor set in which the map is topologically conjugate to the shift $\sigma$ or subshift  $\sigma_A$ since these shift and subshift actually are all chaotic in several senses) and Li-Yorke pairs of these kind of $A$-coupled-expanding maps have full Hausdorff dimension in the invariant set. And the result of [5] on the Hausdorff dimension of Li-Yorke pairs of maps topologically conjugate to a full shift and having a self-similar invariant set is generalized to the case of maps topologically semi-conjugate to some kinds of subshifts. Moreover, Hausdorff dimension of ``chaotic invariant set" for some kinds of $A$-coupled-expanding maps has been  counted.


\vskip0.8cm\noindent
{\bf References}
\small\vskip0.4cm
\begin{itemize}

\item[{[1]}]  Bernd,A., Bernd, K.,  On three definitions of chaos, Nonlinear Dyn. Syst. Theory, 2001,  1(1), 23-37
\item[{[2]}]  Hutchinson, J.,  Fractals and self-similarity, Indiana Univ. math. J. 1981, 30, 271-280
\item[{[3]}]   Li, T.,  Yorke, J., Period three implies chaos, Amer. Math. Monthly, 1975,  82(10): 985-992
\item[{[4]}]  Moran, P. Additive functions of intervals and Hausdorff dimension. Proceeding of the Cambridge Philosophical society, 1946, 42, 15-23
\item[{[5]}] Neuohauserer,J.,  Li-Yorke pairs of full Hausdorff dimension for some chaotic dynamical systems, Math. Bohem. 2010, 135, 3, 279-289
\item[{[6]}] Pesin,Y.,  Weiss, H., On the dimension of deterministic and random Cantor-like sets, Math. Res. Lett. 1994, 1:519-529
\item[{[7]}]  Przytycki, F., Urbanski, M., Conformal fractals- Ergodic Theory Method, 2009,

 www.math.unt.edu/~urbanki/pubook/pu0905517.pdf
\item[{[8]}] Ri, C., Ju, H., Wu, X., Entropy for A-coupled-expanding maps and chaos.2013,  arXiv:1309.
6769v2 [math.DS] 28 Sep 
\item[{[9]}]  Shi, Y., Yu, P., Study on chaos induced by turbulent maps in noncompact sets.
Chaos Solitons Fractals 2006, 28, 1165-1180 
\item[{[10]}]  Shi, Y., Yu, P., Chaos induced by regular snap-back repellers. J. Math. Anal.
Appl. 2008, 337, 1480-1494
\item[{[11]}]  Shi, Y., Ju, H., Chen, G., Coupled-expanding maps and one-sided symbolic
dynamical systems, Chaos Solitons Fractals 2009, 39(5), 2138-2149 
\item[{[12]}] Zhang, X., Shi, Y. : Coupled-expanding maps for irreducible transition matrices.
Int. J. Bifur. Chaos 2010, 20(11), 3769-3783 
\item[{[13]}] Zhang, X., Shi, Y., Chen, G., Some properties of coupled-expanding maps in compact sets.  Proceeding of the American Mathematical Society, 2013, 141(2), 585-595 

\end{itemize}

\end{document}